\documentclass[amstex,11pt, leqno]{article}
\usepackage{amsmath,amsfonts,amsthm,latexsym,amssymb,mathrsfs}
\def\H_0{\mathcal{H}_0(T)}

\def\X{{\mathcal X}}

\def\ind{{\textrm{ind}}}

\def\ind{\textrm{ind}}
\def\X{{\cal X}}
\def\Y{{\cal Y}}
\def\H{{\cal H}}

\newtheorem{df}{Definition}[section]
\newtheorem{thm}[df]{Theorem}
\newtheorem{pro}[df]{Proposition}

\newtheorem{rema}[df] {Remark}
\newtheorem{lem}[df] {Lemma}
\def\sfstp{{\hskip-1em}{\bf.}{\hskip1em}}

\def\enddemo{\qed \endtrivlist}
\expandafter\let\csname enddemo*\endcsname=\enddemo

\def\qedsymbol{\ifmmode\bgroup\else$\bgroup\aftergroup$\fi
  \vcenter{\hrule\hbox{\vrule
height.6em\kern.6em\vrule}\hrule}\egroup}
\def\qed{\ifmmode\else\unskip\nobreak\fi\quad\qedsymbol}

\def\subject#1{\renewcommand{\thefootnote}{}
\footnote{\noindent *Corresponding author. {#1}}}
\begin{document}

\pagestyle{myheadings} \markboth{  }{ \hskip5truecm \rm E. Boasso and M. Amouch  }
\title
{ \bf Generalized Browder's and Weyl's \\ Theorems for 
Generalized Derivations \/}

\author {\normalsize {Enrico  $\hbox{\rm Boasso}^*$ and  Mohamed  Amouch}}

\date{   }

\maketitle\setlength{\baselineskip}{12pt}

\subject{}

\begin{abstract}\noindent Given Banach spaces $\X$ and $\Y$  
and Banach space operators $A\in L(\X)$ and $B\in L(\Y)$,
let $\rho\colon L(\Y,\X)\to L(\Y,\X)$ denote  the generalized derivation
defined by $A$ and $B$, i.e., $\rho (U)=AU-UB$ ($U\in L(\Y,\X)$). 
The main objective of this article is to study Weyl
and  Browder type theorems for $\rho\in L(L(\Y,\X))$. To this end, however,
first the isolated points of the spectrum and the Drazin spectrum of 
$\rho\in L(L(\Y,\X))$ need to be characterized. In addition,
it will be also proved that if $A$ and $B$ are polaroid (respectively 
isoloid), then $\rho$ is polaroid (respectively isoloid).\par
\end{abstract}

\section {\sfstp Introduction}\setcounter{df}{0}
\
\indent In the recent past several authors have studied Weyl and Browder type theorems
and the condition of being polaroid or isoloid for  tensor product and  elementary operators; 
see for example \cite{AD, BDu, BDJ, Du10, DDK, DHK, HK, KD, SK}. Concerning the above mentioned research area,
generalized derivations have been studied mainly for particular classes of operators defined on 
Hilbert spaces; see \cite{CDD, CDDY,Du0, Du110, DK, DK2, F, FB}.  \par

\indent The main objective of this article is to study Weyl and Browder type theorems for generalized derivations in the
context of Banach spaces. In fact, given  two Banach spaces $\X$ and $\Y$, two Banach space operators
$A\in L(\X)$ and $B\in L(\Y)$ and $\rho\in L(L(\Y,\X))$  the generalized derivation defined by $A$ and $B$,
using an approach similar to the one in \cite{BDJ, Du10, DDK},
in section 4 the problem of transferring (generalized) Browder's theorem from $A$ and $B$ and
(generalized) $a$-Browder's theorem from $A$ and $B^*$ to $\rho$ will be studied. Furthermore,
in section 5, when $A$ and $B^*$ are isoloid (respectively $a$-isoloid) operators satisfying generalized Weyl's
(respectively generalized $a$-Weyl's) theorem, necessary and sufficient conditions for
$\rho$ to satisfy  generalized Weyl's (respectively generalized $a$-Weyl's) theorem
will be given; what is more, Weyl's and $a$-Weyl's theorems will be also studied.  In addition, it will be proved that the condition of being polaroid (respectively isoloid) transfers from $A$ and $B$ 
(respectively from $A$ and $B^*$) to $\rho$.\par
\indent However, to this end, in section 3, after having recalled some preliminary definitions and facts in section 2,
the isolated points of the spectrum and the Drazin spectrum of $\rho$ will be fully characterized.

\section {\sfstp Preliminary definitions and facts}\setcounter{df}{0}
\
\indent From now on $\X$ and $\Y$ will denote 
infinite dimensional complex Banach spaces and $L(\Y, \X)$ the algebra of
all bounded linear maps defined on $\Y$ and with values in $\X$; as usual,
$L(\X)=L(\X,\X)$. Given $A\in L(\Y,\X)$, $N(A)\subseteq \Y$ and $R(A)\subseteq \X$ will stand for the null space and the
range of $A$, respectively. In addition, $\X^*$ will denote the dual 
space of $\X$ while $A^*\in L(\X^*)$ will stand for
the adjoint operator of $A\in L(\X)$. Recall that $A\in L(\X)$ is said to be
\it bounded below\rm, if $N(A)=0$ and $R(A)$ is closed. Denote
the \it approximate point spectrum \rm of $A$ by
$\sigma_a(A)=\{\lambda\in \mathbb C \colon A-\lambda \hbox{ is
not bounded below} \}$, where $A-\lambda$ stands for $A-\lambda
I$, $I$ the identity map of $L(\X)$. Let $\sigma(A)$ denote the spectrum of $A$. \par

\indent Recall that $A\in L(\X)$  is said to be  a
\it Fredholm \rm operator if $\alpha(A)=\dim N(A)$ and $\beta(A)=\dim \X/R(A)$
are finite dimensional, in which case its \it index \rm is given by
$$
\ind(A)=\alpha(A)-\beta (A).
$$
If $R(A)$ is closed and $\alpha (A)$ is finite,
then $A\in  L(\X)$ is said to be  \it upper \rm semi-Fredholm\rm,
while if $\alpha (A)$ and $\beta (A)$ are finite and equal, so that the index is zero,
$A$ is said to be a \it Weyl \rm operator.\
These classes of operators generate the
Fredholm or essential spectrum, the upper semi-Fredholm spectrum and the Weyl spectra of $A\in L(\X)$, which will be denoted by
 $\sigma_e(A)$, $\sigma_{SF_+}(A)$ and $\sigma_w(A)$,
 respectively.  
On the other hand, $\Phi(A)$ and $\Phi_+(A)$ will denote the complement
in $\mathbb C$ of the Fredholm spectrum and of the upper semi-Fredholm spectrum of $A$,
respectively.
\par

\indent In addition, the \it Weyl essential approximate point spectrum
\rm of $A\in L(\X)$ is the set $\sigma_{aw}(A)=\{\lambda\in \sigma_a(A)\colon
A-\lambda\hbox{ is not upper semi-Fredholm or } 0<\ind (A-\lambda) \}$,
see \cite{R2}. \par

\indent Recall that the concept of Fredholm operator
has  been generalized. 
An operator $A\in L(\X)$ will be said to be
 \it B-Fredholm\rm, if there exists
$n\in\mathbb N$ for which $R(A^n)$ is closed and the
induced operator $A_n\in L(R(A^n))$ is Fredholm. In a similar way
it is possible to define  upper B-Fredholm operators. Note that if for some $n\in\mathbb N$,
$A_n\in L(R(A^n))$ is Fredholm, then $A_m\in L(R(A^m))$ is
Fredholm for all $m\ge n$; moreover $\ind (A_n)=\ind (A_m)$, for
all $m\ge n$. Therefore, it makes sense to define the index of
$A$ by $\ind (A)=\ind (A_n)$. Recall that $A$ is said to be  \it B-Weyl\rm, if $A$ is  B-Fredholm and
$\ind(A)=0$. Naturally, from this class of operators 
the B-Weyl spectrum of $A\in L(\X)$ can be derived,
which will be denoted by $\sigma_ {BW}(A)$. In addition, set $\sigma_ {SBF_+^-}(A)= \{\lambda\in
\mathbb C\colon A-\lambda  \hbox{ is not upper semi B-Fredholm or
} 0<\ind (A-\lambda)\}$, see \cite{B6}.\par

\indent On the other hand, the \it ascent \rm  (respectively \it descent\rm ) of
$A\in L(\X)$ is the smallest non-negative integer $a$
(respectively $d$) such that $N(A^a)=N(A^{a+1})$ (respectively
$R(A^d)=R(A^{d+1})$); if such an integer does not exist, then
$asc(A)=\infty$ (respectively $dsc(A)=\infty$). The
operator $A$ will be said to be \it Browder\rm,
if it is Fredholm and its ascent and descent are finite. Then,
the Browder spectrum of $A\in L(\X)$ is the set
$\sigma_b(A)=\{\lambda\in \mathbb C\colon A-\lambda\hbox{ is not Browder}\}$. It is well known that
$$
\sigma_e(A)\subseteq \sigma_w(A)\subseteq \sigma_b(A)= \sigma_e(A)\cup acc\hskip .1truecm\sigma (A),
$$
where if $K\subseteq \mathbb C$, then acc $K$ denotes the limit points of $K$
while iso $K$ stands for the isolated points of $K$, i.e.,  iso $K$=$K\setminus$ acc $K$.

\indent In addition, the \it Browder essential  approximate point
spectrum \rm of $A\in L(\X)$ is the set
$\sigma_{ab}(A) =\{\lambda\in \sigma_a(A)\colon \lambda\in \sigma_{aw}(A)
 \hbox{ or } asc(A-\lambda)=\infty\}$, see
 \cite{R2}. It is clear that
$\sigma_{aw}(A)\subseteq \sigma_{ab}(A)\subseteq \sigma_a(A)$.\par

\indent Recall that a Banach space operator $A\in L(\X)$ is said to be Drazin invertible,
if there exists a necessarily unique $B\in L(\X)$ and some $m\in \mathbb N$ such that
$$
A^m=A^mBA, \hskip.3truecm BAB=B, \hskip.3truecm AB=BA,
$$
see for example \cite{D, K}. If $DR(L(\X))=\{ A\in L(\X)\colon
A\hbox{ is Drazin invertible} \}$, then the Drazin spectrum of
$A\in L(\X)$ is the set $\sigma_{DR}(A)=\{\lambda\in C\colon
A-\lambda\notin DR(L(\X)) \}$, see \cite{BS, Bo3}.  It is well known 
that if $A\in L(\X)$ is Drazin invertible, then there is $k\in\mathbb N$ 
such that $\X=R(A^k)\oplus N(A^k)$ and $A_k\in L(R(A^k))$ is invertible, see for example \cite{K}.
In particular, $\sigma_{BW}(A)\subseteq \sigma_{DR}(A)\subseteq \sigma(A)$.
\par

\indent To introduce the definitions of the main notions studied in this work,
some notation is needed. Let $A\in L(\X)$ and denote by $\Pi (A)=
\{\lambda\in \mathbb C\colon
asc(A-\lambda)=dsc(A-\lambda)<\infty\}$ (respectively
$\Pi_0(A)=\{\lambda\in \Pi(A)\colon \alpha(A-\lambda)<\infty\}$)
the set of poles of $A$ (respectively the poles of finite rank of
$A$). Similarly, denote by $\Pi ^a(A)= \{\lambda\in
\hbox{\rm iso }\hskip.1truecm \sigma_a(A)\colon a=asc(A-\lambda)<\infty \hbox{ and
}R(A-\lambda)^{a+1} \hbox{ is closed}\}$ (respectively
$\Pi_0^a(A)= \{\lambda\in\Pi ^a(A)\colon
\alpha(A-\lambda)<\infty\}$) the set of left poles of $A$
(respectively, the left poles of finite rank of $A$). \par

\indent In addition, given $A\in L(\X)$, let $E(A)=\{ \lambda\in \hbox{iso }\sigma (A)\colon 0<\alpha (A-\lambda)\}$
 (respectively $E_0(A)=\{ \lambda\in E(A)\colon \alpha(A-\lambda)<\infty\}$)
the set of eigenvalues of $A$ which are isolated in the spectrum of $A$
(respectively, the eigenvalues of finite multiplicity isolated in $\sigma (A)$).
Similarly, let $E^a(A)=\{ \lambda\in \hbox{iso } \sigma_a (A)\colon 0<\alpha (A-\lambda)\}$
(respectively $E_0^a(A)=\{ \lambda\in E^a(A)\colon \alpha(A-\lambda)<\infty\}$)
the set of eigenvalues of $A$ which are isolated in $\sigma_a(A)$
(respectively  the  eigenvalues of finite multiplicity isolated in $\sigma_a(A)$).
\par

\indent Next the definitions of the main notions studied in this article will be given.\par

\begin{df}\label{def1}Consider a Banach space $\X$ and $A\in L(\X)$. Then it will be said that\par
\noindent \rm (i) \it Browder's theorem holds for $A$, if $\sigma_w(A)=\sigma(A)\setminus \Pi_0(A)$,\par
\noindent \rm (ii) \it generalized Browder's theorem  holds for $A$, if $\sigma_{BW}(A)=\sigma(A)\setminus \Pi(A)$,\par
\noindent \rm (iii) \it $a$-Browder's theorem holds for $A$, if $\sigma_{aw}(A)=\sigma_a(A)\setminus \Pi_0^a(A)$,\par
\noindent \rm (iv) \it generalized $a$-Browder's theorem  holds for $A$, if $\sigma_{SBF_+^-}(A)=\sigma_a(A)\setminus \Pi^a(A)$,\par
\end{df}
\pagestyle{myheadings} \markboth{  }{ \rm \hskip5truecm Browder's and Weyl's Theorems}
\begin{df}\label{def2}Consider a Banach space $\X$ and $A\in L(\X)$. Then it will be said that\par
\noindent \rm (i) \it Weyl's theorem holds for $A$, if $\sigma_w(A)=\sigma(A)\setminus E_0(A)$,\par
\noindent \rm (ii) \it generalized Weyl's theorem  holds for $A$, if $\sigma_{BW}(A)=\sigma(A)\setminus  E(A)$,\par
\noindent \rm (iii) \it $a$-Weyl's theorem holds for $A$, if $\sigma_{aw} (A)=\sigma_a (A)\setminus E^a_0(A)$,\par
\noindent \rm (iv) \it generalized $a$-Weyl's theorem holds for $A$, if $\sigma_{SBF^-_+} (A)=\sigma_a (A)\setminus E^a (A)$.
\end{df}

\indent Finally, recall that an operator $T\in L(\X)$ is said to
have SVEP, the single--valued extension property,  at a (complex)
point $\lambda_0$, if for every open disc $D$ centered
at $\lambda_0$ the only analytic function $f:D\longrightarrow \X$ satisfying $(T-\lambda)f(\lambda)=0$ is the
function $f\equiv 0$. We say that $T$ has SVEP on a subset
$K$ of the complex plane if it has SVEP at every point
of $K$. Trivially, every operator $T$ has SVEP at
points of the resolvent $\rho(A)={\mathbb C}\setminus \sigma(T)$. Also
$T$ has  SVEP at points $\lambda\in \textrm{iso}\hskip.1truecm \sigma(T)$ and
$\lambda\in \textrm{iso}\hskip.1truecm \sigma_a(T)$.
See \cite[Chapters 2-3]{A} for
more information on operators with SVEP.\par

\section {\sfstp Spectra of generalized derivations } \setcounter{df}{0}
\

\indent In this section the Browder spectrum, the Browder essential approximate point spectrum
and the Drazin spectrum of a generalized derivation will be characterized.
To this end, let $\X$ and $\Y$ be two Banach spaces and consider $A\in L(\X)$ and $B\in L(\Y)$.
Let $\rho\colon L(\Y,\X)\to L(\Y,\X)$ be the generalized derivation defined
by $A$ and $B$, i.e., $\rho (U)= AU-UB$, $U\in L(\Y,\X)$. In other words,
$\rho=L_A-R_B$, where if
$S\in L(\X)$ (respectively $S\in L(\Y)$), then $L_S\in L(L(\Y,\X))$ (respectively $R_S\in L(L(\Y,\X))$) is
 the operator defined by left (respectively right) multiplication by $S$, i.e., 
for $U\in L(\Y,\X)$, $L_S(U)=SU$ (respectively $R_S(U)=US$).
In first place, the Browder spectrum of $\rho\in L(L(\Y,\X))$ will be studied.\par

\begin{rema}\label{rem1} \rm Let $\X$ and $\Y$ be two Banach spaces and consider
$A\in L(\X)$ and $B\in L(\Y)$. Let $\rho\in L(L(\Y,\X))$ be the generalized derivation defined by $A$ and $B$.\par
\noindent (i) According to \cite[Corollary 3.4]{E},
$\sigma(\rho)=\sigma (A)-\sigma (B)$. \par
\noindent (ii) It is not difficult to prove
that 
\begin{align*} \hbox{acc }\sigma (\rho)&=(\hbox{acc }\sigma(A)-\sigma (B))\cup
(\sigma (A)-\hbox{acc } \sigma (B))\\
&= (\hbox{acc }\sigma (A)- \hbox{acc } \sigma (B))\cup (\hbox{acc } \sigma (A)-
\hbox{iso } \sigma (B))\cup (\hbox{iso } \sigma (A) -\hbox{acc } \sigma (B)).
\end{align*}

\indent What is more,
$$
\hbox{iso }\sigma (\rho)= (\hbox{iso }\sigma (A)-\hbox{iso } \sigma (B))\setminus 
\hbox{acc } \sigma (\rho).
$$
\noindent (iii) Recall that according to \cite[Corollary 3.4]{E},
$\sigma_e (\rho)= (\sigma (A)-\sigma_e(B))\cup (\sigma_e(A)-\sigma(B))$.
Now well, since $\sigma_b (\rho)=\sigma_e (\rho)\cup \hbox{ \rm acc } \sigma (\rho)$,
according to (ii),
$$
\sigma_b(\rho)= (\sigma (A)-\sigma_b (B))\cup (\sigma_b(A)-\sigma (B)).
$$
\end{rema}
\indent Next the Browder essential approximate point spectrum of $\rho\in L(L(\Y,\X))$
will be characterized. However, first some preparation is needed.\par

\begin{rema}\label{rem32}\rm Let $\X$ and $\Y$ be two Banach spaces and consider $A\in L(\X)$ and $B\in L(\Y)$.
Let $M$ be the two-tuple of commuting operators $M=(L_A,R_B)$,
$L_A$ and $R_B\in L(L(\Y,\X))$. Recall that the \it approximate
 point joint spectrum \rm and the \it upper semi-Fredholm
joint spectrum \rm of $M$ are the sets
 \begin{align*}
&\sigma_{\pi}(M)=\{ (\mu,\nu)\in \mathbb C^2\colon V(A-\mu,
B-\nu) \hbox{ is not bounded below}\}\\
 \hbox{ and                  }\hskip3truecm&\\
&\sigma_{\Phi_+}(M)=\{ (\mu,\nu)\in \mathbb C^2\colon V(A-\mu, B-\nu) \hbox{ is not upper semi-Fredholm}\},\\
\end{align*}
respectively, where $V(A-\mu, B-\nu)\colon L(\Y,\X)\to
L(\Y\X)\times L(\Y,\X)$, $V(A-\mu, B-\nu)(S)= (L_{A-\mu}(S),
R_{B-\nu}$  
$(S))= ((A-\mu)S, S(B-\nu))$. Note that the conditions of being 
bounded below and upper semi-Fredholm  for operators defined between
two diferent Banach spaces is similar to the one given in section 2. Concerning the properties of
these joint spectra, see for example \cite{Bo1,BHW,Sl}.
\end{rema}
\pagestyle{myheadings} \markboth{  }{ \hskip5truecm \rm E. Boasso and M. Amouch  }
\begin{pro}\label{pro6} Let $\X$ and $\Y$ be two Banach spaces and  consider $A\in L(\X)$ and $B\in L(\Y)$. Let $\rho\in L(L(\Y,\X))$ be the 
generalized derivation defined by $A$ and $B$. Then the following statement hold.\par
\noindent \rm (i) \it $\sigma_a (\rho)=\sigma_a(A)-\sigma_a(B^*)$.\par
\noindent \rm (ii) \it  $\sigma_{SF_+} (\rho)= (\sigma_{SF_+} (A)-\sigma_a(B^*))\cup (\sigma_a(A)-\sigma_{SF_+}(B^*))$.\par
\noindent \rm (iii) \it $\hbox{\rm acc }\sigma_a(\rho)= (\hbox{\rm acc }\sigma_a (A)-\sigma_a(B^*))\cup (\sigma_a (A)-\hbox{ \rm acc }\sigma_a
(B^*))$.\par
\noindent \rm (iv) \it $\hbox{\rm iso }\sigma_a (\rho)=(\hbox{\rm iso }\sigma_a (A)-\hbox{ \rm iso }\sigma_a (B^*))\setminus \hbox{ \rm acc }\sigma_a(\rho)$.\par
\noindent \rm (v) \it  $\sigma_{ab} (\rho)= (\sigma_{ab} (A)-\sigma_a(B^*))\cup (\sigma_a(A)-\sigma_{ab}(B^*))$.\par
\end{pro}

\begin{proof} (i). Recall that according to the proof of \cite[Proposition 4.3(i)]{BDJ}, 
$$\sigma_{\pi}(L_A,R_B)=\sigma_a(A)\times \sigma_a(B^*).$$

To prove the statement under consideration, apply the spectral mapping theorem to the polynomial mapping $P\colon \mathbb{C}^2\to\mathbb{C}$,
$P(X,Y)=X-Y$ (\cite[Theorem 2.9]{Sl}).\par

\noindent (ii). According to the proof of \cite[Proposition 4.3(ii)]{BDJ}, 
$$
\sigma_{\Phi_+}(L_A,R_B)=\sigma_{SF_+}(A)\times \sigma_a(B^*)\cup \sigma_a(A)\times \sigma_{SF_+}(B^*).
$$
However, the statement under consideration can be derived applying  the spectral mapping theorem to the 
polynomial mapping $P\colon \mathbb{C}^2\to\mathbb{C}$,
$P(X,Y)=X-Y$ (\cite[Theorem 7]{BHW}).\par

\noindent (iii).-(iv). These statements can be easily deduced from statement (i).\par

\noindent (v). Denote by $\mathbb{S}$ the set $\mathbb{S}=(\sigma_{ab} (A)-\sigma_a(B^*))\cup (\sigma_a(A)-\sigma_{ab}(B^*))$.
Let $\lambda\in\sigma_a (\rho)\setminus\mathbb{S}$. Then, according to what has been proved and to \cite[Corollary 2.2]{R2},
$\lambda\notin \sigma_{SF_+}(\rho)$ and $\lambda\in\hbox{ \rm iso }\sigma_a (\rho)$.
In particular, according to \cite[Theorem 3.16]{A}, $asc (\rho-\lambda)$ is finite.
Therefore, according to \cite[Corollary 2.2]{R2}, $\lambda\in\sigma_a (\rho)\setminus\sigma_{ab}(\rho)$,
Hence, $\sigma_{ab}(\rho)\subseteq \mathbb{S}$.\par
\indent Now let $\lambda\in\sigma_a (\rho)\setminus \sigma_{ab}(\rho)$. Then, according to \cite[Corollary 2.2]{R2} and statement (iv),
there exist $\mu\in \hbox{ \rm iso }\sigma_a (A)$ and $\nu\in \hbox{ \rm iso }\sigma_a (B^*)$
such that $\lambda=\mu-\nu$. In addition, since $\sigma_{SF_+}(\rho)\subseteq \sigma_{ab}(\rho)$, according to statement (ii),
$\mu\notin \sigma_{SF_+}(A)$ and $\nu\notin \sigma_{SF_+}(B^*)$. In particular, according to \cite[Theorem 3.16]{A}
and  \cite[Corollary 2.2]{R2}, $\mu\notin \sigma_{ab}(A)$ and $\nu\notin \sigma_{ab}(B^*)$.
Consequently, $\lambda\in\sigma_a (\rho)\setminus \mathbb{S}$. Hence, $\mathbb{S}\subseteq \sigma_{ab}(\rho)$.
\end{proof}
 
\indent To fully characterize the Drazin spectrum of a generalized derivation $\rho\colon L(\Y,\X)\to L(\Y,\X)$, it is first necessary to
characterize the set of poles  $\Pi (\rho)$ and its complement in the isolated points of the spectrum of $\rho$, i.e.,  $I(\rho)=\hbox{ \rm iso } \sigma (\rho)\setminus \Pi (\rho)$. To this end, however, 
the following result need to be considered.\par
\pagestyle{myheadings} \markboth{  }{ \rm \hskip5truecm Browder's and Weyl's Theorems}

\begin{thm} \label{thm2}Let $\X$ and $\Y$ be two Bananch spaces and suppose that $A\in L(\X)$ and $B\in L(\Y)$ are such that $\sigma (A)=\{ \mu \}$
and $\sigma (B)=\{\nu \}$. Consider $\rho\in L(L(\Y,\X))$, the generalized derivation defined by $A$ and $B$.
Then, the following statements hold.\par
\noindent \rm (i) \it If $\Pi (A)=\{\mu\}$ and $\Pi (B)= \{ \nu \}$,
then $\sigma (\rho)=\Pi (\rho)=\{\mu-\nu\}$.\par

\noindent \rm (ii) \it  If $I(A)= \{ \mu \}$ or $I(B)= \{ \nu \}$, then
$\sigma (\rho)=I(\rho)=\{\mu-\nu\}$. 
 \end{thm}
\begin{proof}Clearly, $\sigma (\rho)=\{\mu-\nu\}$. In addition, 
Note that
$$
\rho-(\mu-\nu)= L_{(A-\mu)}-R_{(B-\nu)}.
$$
\noindent (i) Suppose that $\Pi (A)=\{\mu\}$ and $\Pi (B)= \{ \nu \}$. Since the operators $L_{(A-\mu)}$ and $R_{(B-\nu)}$ are nilpotent (\cite[Remark 3.1(i)]{B}) and commute, 
$\rho-(\mu-\nu)$ is nilpotent, equivalently, $\sigma (\rho)=\Pi (\rho)=\{\mu-\nu\}$ (\cite[Remark 3.1(i)]{B}) .\par

\noindent (ii) Let $x\in \X$ and $f\in \Y^*$. Define the operator $U_{x,f}\in L(\Y,\X)$ as
follows: $U_{x,f}(z)=xf(z)$, where $z\in \Y$. Note that
since $\sigma (A)=\{ \mu \}$, there exist a sequence $(x_n)_{n\in \mathbb N}\subset X$ such that
$\parallel x_n\parallel =1$, for all $n\in \mathbb N$, and $((A-\mu)(x_n))_{n\in\mathbb N}$ converges
to $0$. Suppose that $I(B)= \{ \nu \}$. Since $\sigma (B^*)= I(B^*)= \{ \nu \}$, for each $k\in \mathbb N$, there
is $f_k\in \Y^*$ such that $\parallel ((B-\nu)^*)k(f_k)\parallel =2$ (\cite[Remark 3.1(ii)]{B}).\par

\indent Now well, note that given $k\in\mathbb N$,
$$
(\rho-(\mu-\nu))^k=\sum_{j=1}^k c_{k,j} L_{(A-\mu)^j}R_{(B-\nu)^{k-j}} +(-1)^k R_{(B-\nu)^k},
$$

\noindent where $c_{k,j}=(-1)^{k-j}\frac{k!}{(k-j)!j!}$. However,
using $(U_{x_n,f_k})_{n\in\mathbb N}\subset L(\Y, \X)$, it is not difficult to prove that there is 
$n_0\in \mathbb N$ such that 
$$\parallel (\sum_{j=1}^k c_{k,j} L_{(A-\mu)^j}R_{(B-\nu)^{k-j}})(U_{x_n,f_k})\parallel<1,$$
\noindent for all $n\ge n_0$, $n\in\mathbb N$. Since $\parallel R_{(B-\nu)^k}(U_{x_n,f_k})\parallel=2$, 
$$\parallel (\rho-(\mu-\nu))^k(U_{x_n,f_k})\parallel\ge 1,$$
\noindent $n\ge n_0$. Therefore, since $k\in\mathbb N$ is arbitrary,
$\sigma (\rho)=I(\rho)=\{\mu-\nu\}$ (\cite[Remark 3.1(ii)]{B}).\par

\indent Interchanging $A$ with $B$, it is possible to prove the theorem under the assumption 
$I(A)= \{ \mu \}$.
\end{proof}

\indent Next, the isolated points of the spectrum of a generalized derivation
will be characterized. 

\begin{thm}\label{thm3}Let $\X$ and $\Y$ be two Banach spaces and consider
$A\in L(\X)$ and $B\in L(\Y)$. Let $\rho\in L(L(\Y, \X))$ be the generalized derivation defined by $A$ and $B$. Then, the following statements hold.\par
\noindent \rm (i) \it $I(\rho)= ((I(A)-$ \rm iso \it  $\sigma (B))\cup$ $($\rm iso \it $\sigma (A)- I(B)))\setminus \hbox{ acc }\sigma (\rho)$.\par
\noindent \rm (ii) \it $\Pi (\rho )= (\Pi (A)-\Pi (B))\setminus \sigma_{DR} (\rho)$.
\end{thm}
 \begin{proof}
\pagestyle{myheadings} \markboth{  }{ \hskip5truecm \rm E. Boasso and M. Amouch  }
\indent Let $\lambda\in$ iso $\sigma
(\rho)$. Note that  there exist $n\in\mathbb N$ and finite spectal sets
$\{\mu\}=\{\mu_1,\ldots ,\mu_n\}\subseteq$ iso $\sigma (A)$ and
$\{\nu\}=\{\nu_1,\ldots ,\nu_n\}\subseteq$ iso $\sigma (B)$ such
that $\lambda=\mu_i-\nu_i$ for all $1\le i\le n$ and that if there are $\mu'\in$ iso $\sigma (A)$ and $\nu'\in$
iso $\sigma (B)$ such that $\lambda=\mu'-\nu'$, then there is $i$, $1\le i\le n$, such that
$\mu'=\mu_i$ and $\nu'=\nu_i$. Corresponding to
these spectral sets there are closed subspaces $M_1$, $M_2$ and
$(M_{1i})_{i=1}^n$ of $\X$ invariant for $A$
 and closed subspaces $N_1$, $N_2$ and $(N_{1i})_{i=1}^n $ of $\Y$ invariant for $B$ such that
$\X=M_1\oplus M_2$, $M_1=\oplus_{i=1}^n M_{1i}$,  $\Y=N_1\oplus
N_2$, $N_1=\oplus_{i=1}^n N_{1i}$, $\sigma(A_1)=\{\mu\}$,
$\sigma(A_2)=\sigma(A)\setminus \{\mu\}$,
$\sigma(A_{1i})=\{\mu_i\}$, $\sigma(B_1)=\{\nu\}$,
$\sigma(B_2)=\sigma(B)\setminus \{\nu\}$ and
$\sigma(B_{1i})=\{\nu_i\}$, where $A_1=A\mid_{M_1}$,
$A_2=A\mid_{M_2}$, $A_{1i}=A\mid_{M_{1i}}$, $B_1=B\mid_{N_1}$,
$B_2=B\mid_{N_2}$ and $B_{1i}=B\mid_{N_{1i}}$. In addition, $\rho-\lambda$ is invertible on the closed invariant subspaces
$L(N_2,M_1)$, $L(N_1,M_2)$,
$L(N_2,M_2)$ and $L(N_{1j},M_{1k})$,
$1\le j\neq k\le n$. Moreover, $L(\Y,\X)$ is the
direct sum of these subspaces and $L(N_{1i},M_{1i})$,
$1\le i\le n$. To prove statement (i), define $Z=
((I(A)-$ \rm iso \it  $\sigma (B))\cup$ $($\rm iso \it $\sigma (A)- I(B))) \setminus \hbox{ acc }\sigma (\rho)$.\rm
\par

\indent Let  $\lambda\in I(\rho)$ and consider $n=n(\lambda)\in \mathbb N$ and  $\mu_i\in$ iso $\sigma(A)$
and $\nu_i\in$ iso $\sigma(B)$ such that $\lambda=\mu_i-\nu_i$, $i=1,\ldots ,n$.
Now well, if $\lambda\notin Z$, then for each $i=1,\ldots ,n$, $\mu_i\in \Pi(A)$ and $\nu_i\in \Pi (B)$.
However, according to Theorem \ref{thm2}(i)  and \cite[Remark  3.1]{B}, $\lambda\in \Pi (\rho)$,
which is impossible. \par

\indent Suppose that $\lambda\in Z\subseteq \hbox{ iso }\sigma (\rho)$. Then, there exist $\mu\in$ iso $\sigma(A)$
and $\nu\in$ iso $\sigma(B)$ such that $\lambda=\mu-\nu$ and either $\mu\in I(A)$
or $\nu\in I(B)$. Applying to $\lambda\in Z$ what has been done in the first paragraph of this proof,
there exist an $n=n(\lambda)\in \mathbb N$ and an $i$, $1\le i\le n$, such that
$\mu=\mu_i$ and $\nu=\nu_i$. Therefore, according to Theorem \ref{thm2}(ii) and \cite[Remark  3.1]{B}, $\lambda\in I(\rho)$.\par

\indent To prove statement (ii), apply what has been proved, the fact that $\hbox{iso }\sigma (\rho)= (\hbox{iso }\sigma (A)-\hbox{iso } \sigma (B))\setminus 
\hbox{acc } \sigma (\rho)$ and \cite[Theorem 4]{K}.
\end{proof}

\indent Recall that given a Banach space $\X$ and $A\in L(\X)$, $A$ is said to be \it polaroid, \rm 
if iso $\sigma (A)=\Pi (A)$, equivalently $I(A)=\emptyset$. As a first application of   Theorem \ref{thm3},
it will be proved that this conditon transfers to generalized derivations.\par 

\begin{thm}\label{thm46} Let $\X$ and $\Y$ be two Banach spaces and consider
$A\in L(\X)$ and $B\in L(\Y)$ such that $A$ and $B$ are polaroid operators. Then, the generalized derivation defined by $A$ and $B$
$\rho\in L(L(\Y, \X))$ is polaroid.
\end{thm}
\begin{proof} According to the proof of Theorem \ref{thm3}, if $I(A)=\emptyset=I(B)$,
then $I(\rho)=\emptyset$ and iso $\sigma (\rho)=\Pi (\rho)= (\Pi (A) -\Pi (B))\setminus \hbox{ \rm acc }\sigma (\rho)$.
\end{proof}
\indent In the following theorem, the Drazin spectrum of a generalized derivation will be characterized.

\begin{thm}\label{thm4}Let $\X$ and $\Y$ be two Banach spaces and consider
$A\in L(\X)$ and  $B\in L(\Y)$. Then, if $\rho\colon L(\Y,\X)\to L(\Y,\X)$ is the generalized derivation defined by $A$ and $B$,
$$\sigma_{DR} (\rho)=(\sigma_{DR} (A)-\sigma (B))\cup (\sigma (A)-\sigma_{DR} (B)).$$ 
\end{thm}
\begin{proof} Recall that according to \cite[Theorem 4]{K}, $\sigma_{DR} (\rho)= I(\rho)\cup \hbox{ acc }\sigma (\rho)$.
To conclude the proof, apply Theorem \ref{thm3}(i) and the fact that 
$\hbox{acc }\sigma (\rho)=(\hbox{acc }\sigma(A)-\sigma (B))\cup
(\sigma (A)-\hbox{acc } \sigma (B)$.
\end{proof}

\section {\sfstp Browder's theorems}\setcounter{df}{0}

\
\indent In this section Browder type theorems will be studied  for generalized derivations.
Recall that Browder's theorem (respectivley $a$-Browder's theorem)
is equivalent to generalized Browder's theorem (respectively to generalized
$a$-Browder's theorem), see \cite[Theorems 2.1-2.2]{AZ}. In first place,
(generalized) Browder's theorem will be considered.\par
\pagestyle{myheadings} \markboth{  }{ \rm \hskip5truecm Browder's and Weyl's Theorems}

\begin{thm}\label{thm41}Let $\X$ and $\Y$ be two Banach spaces and consider
$A\in L(\X)$ and  $B\in L(\Y)$ such that (generalized) Browder's theorem holds for 
$A$ and $B$. If $\rho\colon L(\Y,\X)\to L(\Y,\X)$ is the generalized derivation defined by
$A$ and $B$, then the following statements are equivalent.\par
\noindent \rm (i) \it (generalized) Browder's theorem holds for $\rho$. \par
\noindent \rm (ii) \it $(\hbox{\rm acc }\sigma (A)-\sigma (B))\cup (\sigma(A)- \hbox{\rm acc }\sigma (B))\subseteq \sigma_w (\rho)$.\par
\noindent \rm (iii) \it $ \sigma_w (\rho)= (\sigma_w (A)-\sigma (B))\cup(\sigma (A)-\sigma_w (B))$.\par
\noindent \rm (iv) \it $(\hbox{\rm acc }\sigma (A)-\sigma (B))\cup (\sigma(A)- \hbox{\rm acc }\sigma (B))\subseteq \sigma_{BW} (\rho)$.\par
\noindent \rm (v) \it $\sigma_{BW} (\rho)= (\sigma_{BW} (A)-\sigma (B))\cup(\sigma (A)-\sigma_{BW} (B))$.\par
\noindent \rm (vi)\it Given $\lambda\in \sigma (\rho)\setminus\sigma_w (\rho)$, for all $\mu\in \sigma (A)$ and $\nu\in \sigma (B)$
such that $\lambda=\mu-\nu$, $\mu\in \Phi (A)$, $\nu\in \Phi (B)$, $A$ has the SVEP at $\mu$ and $B^*$ has the SVEP at $\nu$.
\end{thm}
 \begin{proof} Recall that according to \cite[Proposition 2]{Ba}, $\rho\in L(L(\Y,\X))$  satisfies Browder's theorem if and only if  acc $\sigma (\rho)\subseteq \sigma_w (\rho)$. 
In particular, according to Remark \ref{rem1}(ii), statements (i) and (ii) are equivalent.\par

\indent Note that according again  to \cite[Proposition 2]{Ba}, Browder's theorem is equivalent to the identity $\sigma_w (\rho)=\sigma_b (\rho)$.
Therefore, according to Remark \ref{rem1}(iii), 
$$
\sigma_b(\rho)= (\sigma_w (A)-\sigma (B))\cup(\sigma (A)-\sigma_w (B)).
$$
\noindent As a result, Browder's theorem holds for $\rho$ if and only if statement (iii) holds.\par
 
\indent According to \cite[Theorem 2.3]{B5}, necessary and sufficient for $\rho\in L(L(\Y, \X))$ to satisfy generalized Browder's theorem is the fact that
 acc $\sigma (\rho)\subseteq \sigma_{BW} (\rho)$. Consequently, according to Remark \ref{rem1}(ii), statements (i) and (iv) are equivalent.\par

\indent Applying again \cite[Theorem 2.3]{B5}, it is not difficult to prove that  generalized Browder's theorem is equivalent to
the fact that $\sigma_{BW} (\rho)= \sigma_{DR} (\rho)$. In particular, according to Theorem \ref{thm4},
$$
\sigma_{DR} (\rho)=(\sigma_{BW} (A)-\sigma (B))\cup (\sigma (A)-\sigma_{BW} (B)).
$$
\noindent Consequently, (generalized) Browder's theorem is equivalent to statement (v).\par

\indent Next suppose that Browder's theorem holds for $\rho\in L(L(\Y,\X))$ and consider $\lambda\in \sigma (\rho)\setminus\sigma_w (\rho)$.
Let $\mu\in \sigma (A)$ and $\nu\in\sigma (B)$ such that $\lambda=\mu-\nu$. Since $\sigma_w (\rho)=\sigma_b(\rho)$,
$\lambda\in\Pi_0 (\rho)\subseteq$ iso $ \sigma (\rho)$. In particular, 
$\mu\in $ iso $\sigma (A)$ and $\nu\in $ iso $\sigma (B)=\hbox{ iso }\sigma (B^*)$,
which implies that $A$ has the SVEP at $\mu$
and $B^*$ has the SVEP at $\nu$. In addition,
since $\sigma_e (\rho)\subseteq\sigma_b(\rho)$, according to \cite[Corollary 3.4]{E},
$\mu\in\Phi (A)$ and $\nu\in \Phi (B)$. \par

\indent On the other hand, suppose that
statement (vi) holds. Let $\lambda\in\sigma (\rho)\setminus\sigma_w (\rho)$
and consider $M_{\lambda}=\{ (\mu, \nu)\in\sigma (A)\times \sigma (B)\colon
\lambda=\mu-\nu\}$. According to \cite[Lemma 4.1]{E} applied to the 
two tuple of operators $M=(L_A ,R_B)\in L(L(\Y, \X))^2$ and the polynomial mapping $P\colon \mathbb{C}^2\to\mathbb{C}$,
$P(X,Y)=X-Y$, $M_{\lambda}$ is a finite set. Let $m\in\mathbb N$
such that $M_{\lambda}=\{ (\mu_i, \nu_i)\colon 1\le i\le m\}$.
Then, according to the paragraphs before \cite[Theorem 4.2]{E}, 
there is $n\in \{0,\ldots ,m\}$ such that:
$$
\hbox{\rm (i) } \mu_i\in \hbox{ \rm iso } \sigma (A)
(1\le i\le n); \hbox{ \rm (ii) if } n<m, \hbox{ then } \nu_i\in \hbox{\rm iso } \sigma (B) (n+1\le i\le m).
$$
Moreover, since ind $(\rho-\lambda)=0$, according to  \cite[Theorem 4.2]{E} applied to
$M$ and $P$,
$$
\sum_{k=n+1}^m (\hbox{dim } X_B(\nu_k))\hbox{ind} (A-\mu_k)=
 \sum_{k=1}^n (\hbox{dim } X_A(\mu_k))\hbox{ind} (B-\nu_k),
$$
where $X_A(\mu_i)$  (respectively  $X_B(\nu_i))$ 
is the spectral subspace of $A$ (respectively $B$)
associated to the isolated point $\mu_i$ $(1\le i\le n)$ (respectively $\nu_i$ $(n+1\le i\le m)$).\par

\indent Note that since $\mu_i\in\Phi (A)\cap\hbox{\rm iso } \sigma (A)$
(respectively $\nu_i\in\Phi (B)\cap\hbox{\rm iso } \sigma (B))$,
$\mu_i\in \Pi_0 (A)$ (respectively $\nu_i\in \Pi_0(B)$) $(1\le i\le n)$
(respectively $(n+1\le i\le m)$). In addition, according to \cite[Theorem 1.52]{Do},
$0<\hbox{dim } X_A(\mu_i)<\infty$ $(1\le i\le n)$ and  $0<  \hbox{dim } X_B(\nu_i)<\infty$
$(n+1\le i\le m)$.\par
\indent Now well, since $\mu_i\in\Phi (A)$ and $A$ has the SVEP at $\mu_i$, according to \cite[Theorem 3.19(i)]{A}, ind$(A-\mu_i)\le 0$
$(n+1\le i\le m)$. Similarly, since $\nu_i\in\Phi (B)$ and $B^*$ has the SVEP at $\nu_i$, according to \cite[Theorem 3.19(ii)]{A}, ind$(B-\nu_i)\ge 0$
$(1\le i\le n)$. Then, using the identity derived from \cite[Theorem 4.2]{E}, it is not
difficult to prove that  ind$(A-\mu_i)= 0$ $(n+1\le i\le m)$ and ind$(B-\nu_i)=0$ $(1\le i\le n)$.
Therefore, $\mu_i\in \Pi_0(A)$ and $\nu_i\in\Pi_0 (B)$ ($1\le i\le m$), which implies that $\lambda\in\Pi_0 (\rho)$
(Remark \ref{rem1}(iii)). As a result, $\sigma_b(\rho)\subseteq \sigma_w (\rho)$, equivalently 
 $\sigma_b(\rho)= \sigma_w (\rho)$.
\end{proof}

\indent Next (generalized) $a$-Browder theorem will be studied.\par
\pagestyle{myheadings} \markboth{  }{ \hskip5truecm \rm E. Boasso and M. Amouch  }
\begin{thm}\label{thm42}Let $\X$ and $\Y$ be two Banach spaces and consider
$A\in L(\X)$ and $B\in L(\Y)$ such that (generalized) $a$-Browder's theorem holds for 
$A$ and $B^*$. If $\rho\colon L(\Y,\X)\to L(\Y,\X)$ is the generalized derivation defined by
$A$ and $B$, then the following statements are equivalent.\par
\noindent \rm (i) \it (generalized) $a$-Browder's theorem holds for $\rho$. \par
\noindent \rm (ii) \it $(\hbox{\rm acc }\sigma_a (A)-\sigma_a (B^*))\cup (\sigma_a (A)- \hbox{\rm acc }\sigma_a (B^*))\subseteq \sigma_{aw} (\rho)$.\par
\noindent \rm (iii) \it $ \sigma_{aw} (\rho)= (\sigma_{aw} (A)-\sigma_a (B^*))\cup(\sigma_a (A)-\sigma_{aw} (B^*))$.\par
\noindent \rm (iv) \it $(\hbox{\rm acc }\sigma_a (A)-\sigma_a (B^*))\cup (\sigma_a (A)- \hbox{\rm acc }\sigma_a (B^*)) \subseteq \sigma_{SBF_+^-} (\rho)$.\par
\noindent \rm (v) \it Given $\lambda\in \sigma_a (\rho)\setminus\sigma_{aw} (\rho)$, for all $\mu\in \sigma_a (A)$ and $\nu\in \sigma_a (B^*)$
such that $\lambda=\mu-\nu$, $\mu\in \Phi_+ (A)$, $\nu\in \Phi_+ (B^*)$, $A$ has the SVEP at $\mu$ and $B^*$ has the SVEP at $\nu$.
\end{thm}
\begin{proof} According to  \cite[Corollary 2.2]{R2} and \cite[Corollary 2.4]{R2}, necessary and sufficient for the $a$-Browder's theorem to hold for $\rho\colon L(\Y,\X)\to L(\Y,\X)$ is that
acc $\sigma_a (\rho)\subseteq \sigma_{aw}(\rho)$. In particular, to prove the equivalence between statements (i) and (ii), apply
Proposition \ref{pro6}(iii).\par
\indent Note that according to \cite[Corollary 2.2]{R2}, $a$-Browder's theorem 
holds for $\rho$ if and only if $\sigma_{aw} (\rho)=\sigma_{ab}(\rho)$. However, according to Proposition \ref{pro6}(v),
$\sigma_{ab}(\rho)=(\sigma_{aw} (A)-\sigma_a (B^*))\cup(\sigma_a (A)-\sigma_{aw} (B^*))$. Consequently, statements (i) and (iii)
are equivalent.\par
\indent Recall that, according to \cite[Theorem 2.8]{B6},  the generalized $a$-Browder's theorem holds for $\rho\in L(L(\Y,\X))$
if and only if acc $\sigma_a (\rho)\subseteq  \sigma_{SBF_+^-} (\rho)$. Therefore, to prove the equivalence between statements (i) and (iv), apply
Proposition \ref{pro6}(iii).\par

 \indent To prove the equivalence between statements (i) and (v), suppose that  $a$-Browder's theorem holds for $\rho$ and
consider $\lambda\in \sigma_a(\rho)\setminus \sigma_{aw} (\rho)$. Let  $\mu\in \sigma_a(A)$ and $\nu\in\sigma_a(B^*)$ such that
$\lambda=\mu-\nu$. Then, according to statement (iii), $\mu\in\sigma_a (A)\setminus\sigma_{ab} (A)$ and 
$\nu\in\sigma_a (B)\setminus\sigma_{ab} (B^*)$. In particular, according to \cite[Theorem 2.1]{R2},  $\mu\in \Phi_+ (A)$, $\nu\in \Phi_+ (B^*)$ and
asc $(A-\mu)$ and asc $(B^*-\nu)$ are finite, which implies that $A$ has the SVEP at $\mu$ and $B^*$ has the SVEP at $\nu$ (\cite[Theorem 3.8]{A}).\par

\indent Next suppose that statement (v) holds and consider $\lambda\in \sigma_a(\rho)
\setminus \sigma_{aw}(\rho)$. Then, for every $\mu\in \sigma_a(A)$ and $\nu\in\sigma_a(B^*)$ such that
$\lambda=\mu-\nu$, $\mu\in \Phi_+ (A)$, $\nu\in \Phi_+ (B^*)$ and $A$ has the SVEP at $\mu$ and $B^*$ has the SVEP at $\nu$.
According to \cite[Theorems 3.16 and 3.23]{A}, $\mu\in$ iso $\sigma_a(A)$, $\nu\in$ iso $\sigma_a(B^*)$ and asc $(A-\mu)$ and  
asc $(B^*-\mu)$ are finite. However, according to \cite[Corollary 2.2]{R2}, $\mu\in \sigma_a(A)\setminus\sigma_{ab}(A)$
and  $\nu\in \sigma_a(B^*)\setminus\sigma_{ab}(B^*)$. Thus, according to Propostition \ref{pro6}(v), $\lambda\notin\sigma_{ab}(\rho)$.
Hence, $\lambda\in \sigma_a(\rho) \setminus \sigma_{ab}(\rho)$, which implies that 
$\sigma_{ab}(\rho)\subseteq \sigma_{aw}(\rho)$, equivalently, $\sigma_{ab}(\rho)=\sigma_{aw}(\rho)$.
\end{proof} 

\section {\sfstp Weyl's theorems}\setcounter{df}{0}

\
\indent Before considering  Weyl type theorems for generalized derivations,
some preparation is needed. Recall first  that given $A\in L(X)$, $X$ a Banach space, $A$ is said to be \it isoloid \rm (respectively \it $a$-isoloid\rm), 
if iso $\sigma (A)=E(A)$ (respectively iso $\sigma_a (A)=E^a(A)$). These conditions transfer to generalized derivations. \par

\begin{pro}\label{pro55} Let $\X$ and $\Y$ be two Banach spaces and consider
$A\in L(\X)$ and $B\in L(\Y)$. Let $\rho\colon L(\Y,\X)\to L(\Y,\X)$ be the generalized derivation defined by
$A$ and $B$. Then, if $A$ and $B^*$ are isoloid (respectively
$a$-isoloid), $\rho\in L(L(\Y,\X))$ is isoloid (respectively  $a$-isoloid).
\end{pro}
\begin{proof} Given $\lambda\in $ iso $\sigma (\rho)$, according to Remark \ref{rem1}(ii), there exist $\mu\in$
iso $\sigma (A)$ and $\nu\in $ iso $\sigma (B)=$ iso $\sigma(B^*)$ such that $\lambda=\mu-\nu$. Since
$A$ and $B^*$ are isoloid, there are $x\in \X$, $x\neq 0$, and $f\in \Y^*$, $f\neq 0$, such that
$A(x)=\mu x$ and $B^*(f)=\nu f$. Consider the operator $U_{x,f}\in L(\Y,\X)$ defined in the proof of Theorem \ref{thm2}(ii).
Then, an easy calculation proves that $\rho (U_{x,f})=\lambda U_{x,f}$. Therefore, iso $\sigma (\rho)= E (\rho)$.\par

\indent A similar argument, using in particular Proposition \ref{pro6}(iv), proves the $a$-isoloid case.
\end{proof}
\pagestyle{myheadings} \markboth{  }{ \rm \hskip5truecm Browder's and Weyl's Theorems}

\indent Next  (generalized) Weyl's theorem for generalized derivations will be studied.
However, first some fact need to be recalled.\par

\begin{rema}\label{rem65}\rm Let $\X$ be a Banach space and consider
$A\in L(\X)$ and $A^*\in L(\X^*)$. Then, the following statements are equivalent.\par
\noindent (i) $A$ is isoloid and generalized Weyl's theorem holds for $A$.\par
\noindent (ii) $A$ is polaroid and generalized Browder's theorem holds for $A$.\par
\noindent (iii) $A^*$ is polaroid and generalized Browder's theorem holds for $A^*$.\par
\noindent (iv) $A^*$ is isoloid and generalized Weyl's theorem holds for $A^*$.\par
\noindent (v) $A$ is polaroid and Weyl's theorem holds for $A$.\par
\noindent (vi) $A$ is polaroid and Browder's theorem holds for $A$.\par
\noindent (vii) $A^*$ is polaroid and  Browder's theorem holds for $A^*$.\par
\noindent (viii) $A^*$ is polaroid and Weyl's theorem holds for $A^*$.\par
\indent The proof can be easily derived from the notions involved in the statements and
from well known results. Although the details are left to the reader, some indications will be given. \par
\indent To prove the equivalence between statements (i) and (ii),
use \cite[Corollary 2.6]{B5}. The equivalence between statements (ii) and (iii) can be proved using \cite[Theorem 2.8(iii)]{A1} and \cite[Remark B]{B00}.
To prove the equivalence between statements (iii) and (iv), apply what has been proved. According to \cite[Theorem 2.1]{AZ},
statements (ii) and (vi) (respectively statements (iii) and (vii)) are equivalents. Finally, according to \cite[Theorem 2.2]{Du},
statements (v) and (vi) (respectively, statements (vii) and (viii)) are equivalent.
\end{rema}

\begin{thm}\label{thm75}Let $\X$ and $\Y$ be two Banach spaces and consider
$A\in L(\X)$ and $B\in L(\Y)$ such that $A$ and $B^*$ are isoloid and generalized Weyl's theorem holds for 
$A$ and $B^*$. If $\rho\colon L(\Y,\X)\to L(\Y,\X)$ is the generalized derivation defined by
$A$ and $B$, then the following statements are equivalent.\par
\noindent \rm (i) \it generalized Weyl's theorem holds for $\rho\in L(L(\Y,\X))$.\par
\noindent \rm (ii) \it (generalized) Browder's theorem holds for $\rho\in L(L(\Y,\X))$.\par
\noindent \rm (iii) \it (generalized) Browder's theorem holds for $\rho^*\in L(L(\Y,\X)^*)$.\par
\noindent \rm (iv) \it generalized Weyl's theorem holds for $\rho^*\in L(L(\Y,\X)^*)$.\par
\noindent \rm (v) \it Weyl's theorem holds for $\rho\in L(L(\Y,\X))$.\par
\noindent \rm (vi) \it Weyl's theorem holds for $\rho^*\in L(L(\Y,\X)^*)$.\par
\end{thm}
\begin{proof} According to Remark \ref{rem65}, $A$ and $B$ are polaroid. Hence,
according to Theorem \ref{thm46}, $\rho\in L(L(\Y,\X))$ is polaroid. Moreover,
according to \cite[Theorem 2.8(iii)]{A1}, $\rho^*\in L(L(\Y,\X)^*)$ is polaroid.
To conclude the proof, apply Remark \ref{rem65}.
\end{proof}

\indent Note that according to Remark \ref{rem65}, the hypothesis in Theorem \ref{thm75} is equivalent to the fact that $A$ and $B^*$ are polaroid operators satisfying Weyl's theorem.
The following lemma will be useful to study $a$-Weyl's theorem.\par

\begin{lem}\label{lem105} Let $\X$ and $\Y$ be two Banach spaces and consider
$A\in L(\X)$ and $B\in L(\Y)$ two operators such that $A$ and $B^*$ are $a$-isoloid. Consider $\rho\colon L(\Y ,\X)\to L(\Y,\X)$
the generalized derivation defined by $A$ and $B$ and let $\lambda\in E^a_0 (\rho)\subseteq \hbox{ \rm iso } \sigma_a (\rho)$. 
If $\mu\in \hbox{ \rm iso } \sigma_a (A)$ and $\nu\in \hbox{ \rm iso } \sigma_a (B^*)$ are such that $\lambda=\mu-\nu$,
then $\mu\in E^a_0 (A)$ and $\nu\in  E^a_0 (B^*)$.
\end{lem} 
\begin{proof}Note that according to Proposition \ref{pro55}, $\rho\in L(L(\Y,\X))$ is $a$-isoloid.
In particular, according to Proposition \ref{pro6}(iv), if $\lambda\in \hbox{ iso }\sigma_a (\rho)=E^a (\rho)$,
then there are $\mu\in \hbox{ \rm iso }\sigma_a (A)=  E^a (A)$ and $\nu\in \hbox{ \rm iso }\sigma_a (B^*)= E^a (B^*)$ such that $\lambda=\mu-\nu$.
Suppose that $\lambda\in E^a_0 (\rho)$. To prove the Lemma,  it is enough to show that 
dim $N(A-\mu)$ and dim $N(B^*-\nu)$ are dinite dimensional.\par
\indent If  dim $N(A-\mu)$ is not finite dimensional, then there is a sequence of linearly  independent vectors $(x_n)_{n\in\mathbb N}\subseteq N(A-\mu)$ such that
$\parallel x_n\parallel=1$. Let $f\in N(B^*-\nu)$, $f\neq 0$ and consider the sequences of operators $(U_{x_n,f})_{n\in\mathbb N}\subset L(\Y,\X)$ (see the proof of Theorem \ref{thm2}(ii)). 
Then, it is not difficult to prove that $(U_{x_n,f})_{n\in\mathbb N}\subset L(\Y,\X)$ is a sequence of linearly independent operators such that
$(U_{x_n,f})_{n\in\mathbb N}\subset N(\rho-\lambda)$, which is impossible for $\lambda\in E^a_0 (\rho)$.\par
\indent A similar argument proves that if dim $N(B^*-\nu)$ is not finite dimensional, then $\lambda\notin E^a_0 (\rho)$.
\end{proof}

\indent Next $a$-Weyl's theorem will be considered.\par
\pagestyle{myheadings} \markboth{  }{ \hskip5truecm \rm E. Boasso and M. Amouch  }
\begin{thm}\label{thm185} Let $\X$ and $\Y$ be two Banach spaces and consider two operators
$A\in L(\X)$ and $B\in L(\Y)$ such that $A$ and $B^*$ are $a$-isoloid and $a$-Weyl's theorem holds for 
$A$ and $B^*$. If $\rho\colon L(\Y,\X)\to L(\Y,\X)$ is the generalized derivation defined by
$A$ and $B$, then the following statements are equivalent.\par
\noindent \rm (i) \it $a$-Weyl's theorem holds for $\rho\in L(L(\Y,\X))$.\par
\noindent \rm (ii) \it $a$-Browder's theorem holds for $\rho\in  L(L(\Y,\X))$.
\end{thm}
\begin{proof}  According to \cite[Corollary 2.2]{R2} and \cite[Corollary 2.4]{R2}, statement (i) implies
statement (ii).\par
\indent On the other hand, since $\Pi^a_0(\rho)\subseteq E^a_0 (\rho)$, if statement (ii) holds,
then to prove statement (i), it  is enough to show that $E^a_0 (\rho)\subseteq \Pi^a_0(\rho)$.
Let $\lambda\in E^a_0 (\rho)\subseteq \hbox{ iso } \sigma_a (\rho)$. Then,
according to Proposition \ref{pro6}(iv), there exist $\mu\in  \hbox{ iso } \sigma_a (A)$ and 
$\nu\in  \hbox{ iso } \sigma_a (B^*)$ such that $\lambda=\mu-\nu$. In particular,
according to Lemma  \ref{lem105}, $\mu\in E^a_0 (A)=\sigma_a (A)\setminus \sigma_{aw} (A)$ and 
$\nu\in E^a_0 (B^*)=\sigma_a (B^*)\setminus \sigma_{aw} (B^*)$. Consequently,
according to Theorem \ref{thm42}(iii), $\lambda\notin \sigma_{aw} (\rho)=\sigma_{ab} (\rho)$.
Hence, $\lambda\in \Pi^a_0 (\rho)$ (\cite[Corollary 2.2]{R2}).
\end{proof}

\indent In the following theorem generalized $a$-Weyl's theorem for 
generalized derivations will be studied.\par

\begin{thm}\label{thm85}Let $\X$ and $\Y$ be two Banach spaces and consider two operators
$A\in L(\X)$ and $B\in L(\Y)$ such that $A$ and $B^*$ are $a$-isoloid and generalized $a$-Weyl's theorem holds for 
$A$ and $B^*$. If $\rho\colon L(\Y,\X)\to L(\Y,\X)$ is the generalized derivation defined by
$A$ and $B$, then the following statements are equivalent.\par
\noindent \rm (i) \it generalized $a$-Weyl's theorem holds for $\rho\in L(L(\Y,\X))$.\par
\noindent \rm (ii) \it generalized $a$-Browder's theorem holds for $\rho\in  L(L(\Y,\X))$ and 
$\sigma_{SBF^-_+ }(\rho)= (\sigma_a(A)-\sigma_{SBF^-_+ }(B^*))\cup (\sigma_{SBF^-_+ }(A)-\sigma_a (B^*))$.
\end{thm}
\begin{proof} Suppose that statement (i) holds. Then, generalized $a$-Browder's theorem holds for $\rho$ (\cite[Corollary 3.3]{B6}).\par

Let $\lambda\in \sigma_a (\rho)\setminus \sigma_{SBF_+^-} (\rho)= E^a (\rho)=\hbox{ iso }\sigma_a (\rho)$.
According to Proposition \ref{pro6}(iv), $\lambda=\mu-\nu$, where $\mu\in  \hbox{ iso }\sigma_a (A)$ and $\nu\in \hbox{ iso }\sigma_a (B^*)$.
However, since for every $\mu$ and $\nu$ such that $\lambda=\mu-\nu$, $\mu\hbox{ iso }\sigma_a (A)= E^a (A)=\sigma_a (A)\setminus \sigma_{SBF_+^-} (A)$
and $\nu\in \hbox{ iso }\sigma_a (B^*)= E^a (B^*)=\sigma_a (B^*)\setminus \sigma_{SBF_+^-} (B^*)$,
 $\lambda \in \sigma_a (\rho)\setminus ( (\sigma_a(A)-\sigma_{SBF^-_+ }(B^*))\cup (\sigma_{SBF^-_+ }(A)-\sigma_a (B^*)))$.
Consequently, $(\sigma_a(A)-\sigma_{SBF^-_+ }(B^*))\cup (\sigma_{SBF^-_+ }(A)-\sigma_a (B^*))\subseteq \sigma_{SBF^-_+ }(\rho)$.\par

\indent Next suppose that $\lambda\in \sigma_a (\rho)\setminus ((\sigma_a(A)-\sigma_{SBF^-_+ }(B^*))\cup (\sigma_{SBF^-_+ }(A)-\sigma_a (B^*)))$.
Note that according to  \cite[Lemma 2.12]{B6}, acc $\sigma_a (A)\subseteq  \sigma_{SBF^-_+ }(A)$ and
acc $\sigma_a (B^*)\subseteq  \sigma_{SBF^-_+ }(B^*)$. Thus, according to Proposition \ref{pro6}(iii), $\lambda\in $ iso $\sigma_a (\rho)= E^a (\rho)
= \sigma_a (\rho)\setminus \sigma_{SBF_+^-} (\rho)$. Hence, $\sigma_{SBF_+^-} (\rho)\subseteq (\sigma_a(A)-\sigma_{SBF^-_+ }(B^*))\cup (\sigma_{SBF^-_+ }(A)-\sigma_a (B^*))$. \par

\indent Suppose that statement (ii) holds. Since $\Pi^a (\rho)\subseteq E^a (\rho)$, according to \cite[Corollary 3.2]{B6},
it is enough to prove that $E^a (\rho)\subseteq \Pi^a (\rho)$. Let $\lambda\in E^a (A)$. Then, 
according to Proposition \ref{pro6}(iv), there exist  $\mu\hbox{ iso }\sigma_a (A)$ and  $\nu\in \hbox{ iso }\sigma_a (B^*)$
such that $\lambda=\mu-\nu$. However, as in the first paragraph of the proof, $\mu\in \sigma_a (A)\setminus \sigma_{SBF_+^-} (A)$
and $\nu\in\sigma_a (B^*)\setminus \sigma_{SBF_+^-} (B^*)$. Therefore, $\lambda\in\sigma_a (\rho) \setminus \sigma_{SBF_+^-} (\rho)=\Pi^a (\rho)$.
\end{proof}

\vskip.5truecm
\noindent Enrico  Boasso\par
\noindent e-mail: enrico\_odisseo@yahoo.it\par
\vskip.3truecm
\noindent Mohamed  Amouch\par
\noindent Department of Mathematics and Informatic\par
\noindent Chouaib Doukkali Universty\par
\noindent Faculty of Sciences\par
\noindent Eljadida\par
\noindent Morocco\par
\noindent e-mail: amouch.m@ucd.ac.ma

\end{document}